\documentclass[a4paper,12pt]{amsart}
\usepackage{amsmath, amssymb, amsthm, xcolor,enumerate}
\usepackage{verbatim}

\newcounter{theorem}
\newtheorem{theorem}[theorem]{Theorem}
\newtheorem{lemma}[theorem]{Lemma}
\newtheorem{proposition}[theorem]{Proposition}

\theoremstyle{definition}

\newtheorem*{remark*}{Remark}

\newcommand{\centredrelation}[2]{{}_{\phantom{#2}}#1_{#2}}
\newcommand{\capprox}[1]{\centredrelation{\approx}{#1}}

\numberwithin{equation}{section}

\newcommand{\e}{\epsilon}
\newcommand{\tr}{\mathrm{tr}}

\title{Classifying maps into uniform tracial sequence algebras}

\author[J.\ Castillejos]{Jorge Castillejos}
\address{\hskip-\parindent Jorge Castillejos, Institute of Mathematics, Polish Academy of Sciences, ul. {\'S}niadeckich 8, 00-656 Warszawa, Poland}
\email{jcastillejoslopez@impan.pl}
\author[S.\ Evington]{Samuel Evington}
\address{\hskip-\parindent Samuel Evington, Mathematical Institute, University of Oxford, Oxford, OX2 6GG, UK.}
\email{Samuel.Evington@maths.ox.ac.uk}
\author[A.\ Tikuisis]{Aaron Tikuisis}
\address{\hskip-\parindent Aaron Tikuisis, Department of Mathematics and Statistics, University of Ottawa, Ottawa, K1N 6N5, Canada.}
\email{aaron.tikuisis@uottawa.ca}
\author[S.\ White]{Stuart White}
\address{\hskip-\parindent Stuart White, Mathematical Institute, University of Oxford, Oxford, OX2 6GG, UK.}
\email{stuart.white@maths.ox.ac.uk}
\thanks{Research partially supported by: European Research Council Consolidator Grant 614195 RIGIDITY (JC); long term structural funding -- a Methusalem grant of the Flemish Government (JC); EPSRC grant EP/R025061/2 (SE, SW); an NSERC discovery grant (AT)}

\begin{document}
\maketitle
\begin{abstract}
	We classify $^*$-homomorphisms from nuclear $C^*$-algebras into uniform tracial sequence algebras of nuclear $\mathcal Z$-stable $C^*$-algebras via tracial data.
\end{abstract}

\renewcommand*{\thetheorem}{\Alph{theorem}}
\section*{Introduction}

Over the last 10 years, the application of von Neumann techniques has been a major theme in the structure theory of simple nuclear $C^*$-algebras through the pioneering work of Matui and Sato \cite{MS12,MS14}.  A starting point for a number of these applications is the following well-known consequence of Connes' revolutionary work on the characterisation of hyperfinite von Neumann algebras \cite{Co76}: maps from a separable nuclear $C^*$-algebra $A$  into $\mathcal R^\omega$ (the ultrapower of the hyperfinite II$_1$ factor) are classified up to unitary equivalence by the trace they induce on $A$ (see for example \cite[Proposition 2.1]{CGNN13}).  Most recently, this result played a key role in Schafhauser's breakthrough new approach to the classification of monotracial separable nuclear $C^*$-algebras which absorb the universal uniformly hyperfinite algebra tensorially (\cite{Sch:Ann}).   

For a nuclear $C^*$-algebra $B$ with a unique trace, Connes' theorem allows us to view $\mathcal R^\omega$ as a tracial ultrapower of $B$.  When $B$ has multiple traces a somewhat different reduced product construction is needed in order to be able to handle them all uniformly. This led to the uniform tracial ultrapower $B^\omega$, 
formalised in \cite{CETWW}. This, and its precursor in terms of ultraproducts of $W^*$-bundles (\cite{BBSTWW,Oz13}), has been a crucial tool in recent developments.  In particular, in our recent joint work with Winter (\cite{CETWW}), we introduced a new tool --- complemented partitions of unity (CPoU) --- for studying these ultraproducts, and used this to show that Jiang--Su stability and finite nuclear dimension are equivalent in the Toms--Winter conjecture (\cite[Theorem A]{CETWW}). 

The principal purpose of this paper is to use the new techniques in \cite{CETWW} to classify maps from separable nuclear $C^*$-algebras into uniform tracial sequence algebras of $\mathcal Z$-stable nuclear $C^*$-algebras, analogous to the consequence of Connes' result for maps into $\mathcal R^\omega$. In the theorem which follows, the uniform tracial sequence algebra $B^\infty$ associated to $B$ is the $C^*$-algebra of bounded sequences in $B$, modulo those converging to zero uniformly over all trace norms. So classifying maps from $A$ to $B^\infty$ up to unitary equivalence is a way of encoding a classification of uniform trace norm approximately multiplicative maps from $A$ into $B$ up to approximate unitary equivalence in uniform trace norm.

\begin{theorem}\label{thm:MainThm}
Let $A$ be a separable nuclear $C^*$-algebra and let $B$ be a separable nuclear $\mathcal Z$-stable $C^*$-algebra with $T(B)$ compact and non-empty.
Then for any continuous affine function $\alpha:T(B^\infty) \to T(A)$, there exists a $^*$-ho\-mo\-mor\-phism $\phi:A \to B^\infty$ which induces $\alpha$.
Moreover, $\phi$ is unique up to unitary equivalence.
\end{theorem}

 Just as the classification of embeddings from separable nuclear $C^*$-algebras into $\mathcal R^\omega$ is vital in \cite{Sch:Ann}, Theorem \ref{thm:MainThm} will form the starting point of the forthcoming joint work of Carri\'on, Gabe, Schafhauser and the last two named authors which will give an abstract approach to the classification of simple separable unital nuclear Jiang--Su stable $C^*$-algebras satisfying the UCT (\cite{CGSTW}).  Our reason for setting up Theorem \ref{thm:MainThm} with the uniform trace norm sequence algebra $B^\infty$ as opposed to the uniform tracial ultrapower $B^\omega$ is so it can be applied exactly as written in \cite{CGSTW}.\footnote{One can obtain an ultrapower version of Theorem \ref{thm:MainThm} by working with an ultrapower $B^\omega$ in place of the sequence algebra $B^\infty$ throughout the paper.}

The role of nuclearity and $\mathcal Z$-stability of $B$ in Theorem \ref{thm:MainThm} is to obtain CPoU from \cite[Theorem I]{CETWW} --- the main technical result of that work.  Although the defining property of CPoU is in the form of the existence of certain partitions of unity for trace spaces, its principal consequence is a local to global tracial type-satisfaction process for $B^\infty$. In less fancy language, this means that if (suitable) properties hold approximately in trace in each tracial GNS-representation of $B^\infty$, then they hold exactly in $B^\infty$.  This gives $B^\infty$ a von Neumann algebra-like flavour (though it is certainly not a von Neumann algebra).  Theorem \ref{thm:MainThm} is obtained in this fashion: we glue together the classification of maps from separable nuclear $C^*$-algebras into finite von Neumann algebras from Connes' theorem over all traces using CPoU.  

We also record in Proposition \ref{unitaryexp}, another, somewhat easier, application of CPoU for use in \cite{CGSTW}, which showcases another von Neumann algebra-like property: every unitary in $B^\infty$ is an exponential.

\subsection*{Acknowledgements} The authors would like to thank the referee for their careful reading of earlier versions of this paper and helpful comments.

\numberwithin{theorem}{section}
\section{Preliminaries}

Let $B$ be a $C^*$-algebra.
We let $T(B)$ denote the set of tracial states (which we abbreviate as ``traces'') on $B$.
For $\tau \in T(B)$, we define the associated 2-seminorm on $B$ by
\begin{equation} \|b\|_{2,\tau} := \sqrt{\tau(|b|^2)}. \end{equation}
We let $\pi_\tau:B \to \mathcal B(\mathcal H_\tau)$ be the GNS representation associated to $\tau$, and continue to use $\|\cdot\|_{2,\tau}$ to denote the induced 2-norm on $\pi_\tau(B)''$.

Define the uniform tracial sequence algebra
\begin{equation} B^\infty:=\ell^\infty(B)/\{(b_n)_{n=1}^\infty: \lim_{n\to\infty} \sup_{\tau \in T(B)} \|b_n\|_{2,\tau}=0\}. \end{equation}
We will typically use representative sequences in $\ell^\infty(B)$ to denote elements of $B^\infty$.

The ultraproduct versions of these sequence algebras are obtained using a free ultrafilter $\omega\in \beta\mathbb N\setminus\mathbb N$ in place of $\infty$, and many of their basic properties are the same.  For example, when $B$ is separable and $T(B)$ is non-empty and compact, $B^\infty$ is unital, with the unit represented by an approximate unit $(e_n)_{n=1}^\infty$ for $B$ in just the same way as the ultrapower version of this result (\cite[Proposition 1.11]{CETWW}).

Given a sequence $(\tau_n)_{n=1}^\infty$ in $T(B)$ and a free ultrafilter $\omega$, define the associated \emph{limit trace} $\tau:B^\infty \to \mathbb C$ by
\begin{equation} \tau((b_n)_{n=1}^\infty) := \lim_{n\to\omega} \tau_n(b_n). \end{equation}
We let $T_\infty(B)$ denote the set of all limit traces on $B^\infty$. 

On $B^\infty$, one has a uniform $2$-norm\footnote{To see it is a norm, use \eqref{eq:2uNormFormula}.} $\|\cdot\|_{2,T_\infty(B)}$ given by
\begin{equation} \|b\|_{2,T_\infty(B)}:=\sup_{\tau \in T_\infty(B)} \|b\|_{2,\tau},\quad b\in B^\infty. \end{equation}
More explicitly, for $b=(b_n)_{n=1}^\infty \in B^\infty$, one has
\begin{equation} \label{eq:2uNormFormula}
\|b\|_{2,T_\infty(B)}=\limsup_n \sup_{\tau \in T(B)} \|b_n\|_{2,\tau}. \end{equation}

In this paper we will frequently use Kirchberg's $\e$-test, which appears as \cite[Lemma A.1]{Kir06}.
However, we need a slightly different version, as we work with ``sequence algebras'' rather than ultrapowers; we state here the version we need, and remark that the proof is nearly identical to that of \cite[Lemma A.1]{Kir06}.

\begin{lemma}[{Kirchberg's $\e$-test, (\cite{Kir06})}]
\label{lem:EpsTest}
	Let $(X_n)_{n=1}^\infty$ be a sequence of non-empty sets and for $k,n\in\mathbb N$, let $f^{(k)}_n:X_n\rightarrow [0,\infty]$ be a function.  Define functions $f^{(k)}:X_1\times X_2 \times \cdots \to [0,\infty]$ by
	\begin{equation} 
	\label{eq:EpsTest1}
	f^{(k)}(x_1,x_2,\dots) := \limsup_{n} f^{(k)}_n(x_n). \end{equation}
	If for every $\epsilon>0$ and $k_0\in\mathbb N$ there exists $x\in X_1\times X_2 \times \cdots$ such that $f^{(k)}(x)<\epsilon$ for $k=1,\dots,k_0$, then there exists $y\in X_1 \times X_2 \times \cdots$ such that $f^{(k)}(y)=0$ for all $k\in\mathbb N$.
\end{lemma}

We next record a useful fact, which is a direct consequence of Cuntz and Pedersen's investigation of traces on $C^*$-algebras (\cite[Proposition 2.7]{CP79}) and the fact that the weak$^*$-continuous functionals on the dual of a Banach space correspond precisely to elements of the Banach space.

\begin{proposition}[{Cuntz--Pedersen}]
	\label{prop:CP}
	Let $A$ be a $C^*$-algebra and let $f:T(A) \to \mathbb R$ be a continuous affine function.
	Then there is a self-adjoint element $a \in A$ such that
	\begin{equation} \tau(a)=f(\tau),\quad \tau \in T(A). \end{equation}
\end{proposition}

In this paper, we will be using CPoU as a property of the uniform tracial sequence algebra $B^\infty$. In \cite[Definition 3.1]{CETWW}, CPoU was originally defined as a property of the uniform tracial ultrapower $B^\omega$, but standard methods allow it to be rephrased as a local property of $B$ instead: see \cite[Proposition 3.2]{CETWW}. 
The same methods allow it to be rephrased as a property of $B^\infty$, analogous to the original definition for $B^\omega$, as recorded in the lemma below. For the purposes of this paper, one can take this as the definition of CPoU.

\begin{lemma}
\label{lem:CPoU-Infinity}
Let $B$ be a separable $C^*$-algebra with $T(B)$ non-empty and compact.
If $B$ has CPoU, then for any $\|\cdot\|_{2,T_\infty(B)}$-separable subset $S$ of $B^\infty$, any $\delta>0$, and any $a_1,\dots,a_k \in (B^\infty)_+$ satisfying
\begin{equation} \min\{\tau(a_1),\dots,\tau(a_k)\} < \delta, \quad \tau \in T_\infty(B), \end{equation}
there exist orthogonal projections $e_1,\dots,e_k \in B^\infty \cap S'$ which sum to $1_{B^\infty}$, such that
\begin{equation} \tau(a_ie_i) \leq \delta\tau(e_i), \quad i=1,\dots,k,\ \tau \in T_\infty(B). \end{equation}
\end{lemma}

We will access CPoU through one of the main technical results of \cite{CETWW}, which we recall below. Note that for unital $C^*$-algebras the tracial state space $T(B)$ is automatically compact.

\begin{theorem}
\label{thm:CPoU}
\cite[Theorem I]{CETWW}
Let $B$ be a separable, nuclear, $\mathcal Z$-stable $C^*$-algebra with $T(B)$ compact and non-empty.
Then $B$ has CPoU.
\end{theorem}

\section{Results}

We start by recording an application of CPoU regarding unitaries in uniform tracial sequence algebras of $\mathcal Z$-stable nuclear $C^*$-algebras for use in \cite{CGSTW}. 

\begin{proposition}\label{unitaryexp}
Let $B$ be a separable  $C^*$-algebra with CPoU and $T(B)$ compact and non-empty. Let $S$ be a $\|\cdot\|_{2,T_\infty(B)}$-separable subset of $B^\infty$ closed under taking adjoints.
Then every unitary $u \in B^\infty \cap S'$ can be written as an exponential $u = e^{\pi i h}$ for some self-adjoint $h \in B^\infty \cap S'$ of norm at most $1$.
In particular this holds whenever $B$ is separable, unital, nuclear and $\mathcal Z$-stable with $T(B)\neq\emptyset$.
\end{proposition}
\begin{proof}
The final sentence of the proposition follows from the rest by Theorem \ref{thm:CPoU}.   Fix $\e>0$.
By Kirchberg's $\e$-test (Lemma \ref{lem:EpsTest}),\footnote{Since $S$ is $\|\cdot\|_{2,T_\infty(B)}$-separable, testing that a sequence $(b_n) \subset \ell^\infty(B)$ represents an element of $B^\infty \cap S'$ requires only countably many constraints. Indeed, $B^\infty \cap S' = B^\infty \cap S_0'$ for any countable $\|\cdot\|_{2,T_\infty(B)}$-dense subset $S_0 \subset S$, as multiplication is jointly $\|\cdot\|_{2,T_\infty(B)}$-continuous on $\|\cdot\|$-bounded sets.} it suffices to prove that there exists a self-adjoint $h \in B^\infty \cap S'$ of norm at most $1$ such that
\begin{equation} \|u-e^{\pi i h}\|_{2,T_\infty(B)} \leq \e. \end{equation}
For each $\tau \in T(B^\infty)$, using Borel functional calculus, there exists a self-adjoint $x \in \pi_\tau(B^\infty \cap S')''$ of norm at most $1$ such that $\pi_\tau(u)=e^{\pi i x}$. 
By the Kaplansky density theorem, we may approximate $x$ by a self-adjoint contraction in $\pi_\tau(B^\infty \cap S')$, which can then be lifted to a self-adjoint contraction $h_\tau \in B^\infty \cap S'$ such that
\begin{equation} \|u-e^{\pi i h_\tau}\|_{2,\tau} < \e. \end{equation}
Set $a_\tau:=|u-e^{\pi i h_\tau}|^2 \in (B^\infty)_+$, so that $\tau(a_\tau)<\e^2$.
By continuity and compactness, there exist $\tau_1,\dots,\tau_k \in T(B^\infty)$ such that for every $\tau \in T(B^\infty)$,
\begin{equation} \min\{\tau(a_{\tau_1}),\dots,\tau(a_{\tau_k})\} < \e^2. \end{equation}
Using CPoU as in Lemma \ref{lem:CPoU-Infinity}, there exists a partition of unity consisting of projections $p_1,\dots,p_k \in B^\infty \cap \{u,h_{\tau_1},\dots,h_{\tau_k}\}' \cap S'$ such that
\begin{equation} \tau(p_ja_{\tau_j}) \leq \e^2\tau(p_j), \quad \tau \in T_\infty(B),\ j=1,\dots,k. \end{equation}
Set $h:=\sum_{j=1}^k p_jh_{\tau_j} \in B^\infty \cap S'$.
Since the $p_i$ are orthogonal and commute with the self-adjoint contractions $h_{\tau_j}$, this is a self-adjoint contraction.
We note that for $j=1,\dots,k$,
\begin{equation} \label{eq:unitaries}
p_je^{\pi i h} = p_je^{\pi i h_{\tau_j}}. 
\end{equation}
Using this, for $\tau \in T_\infty(B)$, we compute:
\begin{align}
\tau(|u-e^{\pi i h}|^2)
&= \sum_{j=1}^k \tau(p_j|u-e^{\pi i h}|^2) \\
\notag
& = \sum_{j=1}^k \tau(p_j|u-e^{\pi i h_{\tau_j}}|^2) \\
\notag
&= \sum_{j=1}^k \tau(p_ja_j) \\
\notag 
&\leq  \sum_{j=1}^k \e^2\tau(p_j) = \e^2. 
\qedhere
\end{align}
\end{proof}

We next turn to the uniqueness aspect of Theorem \ref{thm:MainThm}. Recall the by-now well-known consequence of Connes' characterisation of hyperfiniteness from \cite{Co76}, that if $A$ is separable and nuclear and $\mathcal M$ is a finite von Neumann algebra, then $^*$-homomorphisms $\phi,\psi:A\to \mathcal M$ are strong$^*$-approximately unitary equivalent if and only if $\tau\circ\phi=\tau\circ\psi$ for all $\tau\in T(\mathcal M)$ (see \cite[Proposition 2.1]{CGNN13}, for example\footnote{In the statement of \cite[Proposition 2.1]{CGNN13}, the codomain $\mathcal M$ is required to be countably decomposable; however, this hypothesis is not needed in the proof. It might also be noted that in our application, in the proof of Theorem \ref{thm:Uniqueness}, the codomain $\pi_\tau(B^\infty)''$ is countably decomposable as it has a faithful trace.}).

\begin{theorem}
\label{thm:Uniqueness}
Let $A$ be a separable nuclear $C^*$-algebra and let $B$ be a separable $C^*$-algebra with CPoU and with $T(B)$ compact and non-empty. If $\phi,\psi:A \to B^\infty$ are $^*$-homomorphisms such that $\tau\circ \phi = \tau\circ \psi$ for all $\tau \in T(B^\infty)$ then $\phi$ and $\psi$ are unitarily equivalent. 
\end{theorem}

\begin{proof}
Fix $\e>0$ and a finite set $\mathcal F\subset A$.
Since $A$ is separable, by using Kirchberg's $\e$-test (Lemma \ref{lem:EpsTest}), it suffices to prove that there is a unitary $u \in B^\infty$ such that 
\begin{equation}
\label{eq:Uniqueness0}
\|\phi(x)-u^*\psi(x)u\|_{2,T_\infty(B)} \leq \e, \quad x \in \mathcal F. \end{equation}
Set
\begin{equation} \eta := \frac{\e}{\sqrt{|\mathcal F|}}. \end{equation}
Fix, for the moment, a trace $\tau \in T(B^\infty)$, and recall that $\pi_\tau:B^\infty \to \pi_\tau(B^\infty)''$ is the corresponding GNS representation.
Then since $A$ is nuclear and $\pi_\tau\circ \phi, \pi_\tau\circ \psi:A \to \pi_\tau(B^\infty)''$ agree on the traces of $\pi_\tau(B^\infty)''$, it follows that these maps are strong$^*$-approximately unitarily equivalent. By Kaplansky's density theorem, the unitaries implementing this can be taken from $\pi_\tau(B^\infty)$ (as done in the proof of Proposition \ref{unitaryexp}). Since the strong$^*$-topology is given by $\|\cdot\|_{2,\tau}$ on bounded sets, it follows that there exists a unitary $u_\tau \in B^\infty$ such that
\begin{equation} 
\|\phi(x)-u_\tau^*\psi(x)u_\tau\|_{2,\tau} < \eta, \quad x \in \mathcal F.
\end{equation}
Set
\begin{equation}
\label{eq:Uniqueness1}
a_\tau:= \sum_{x \in \mathcal F} |\phi(x)-u_\tau^*\psi(x)u_\tau|^2 \in (B^\infty)_+, \end{equation}
so that $\tau(a_\tau) < |\mathcal F|\eta^2 = \e^2$.

By continuity and compactness, there exist $\tau_1,\dots,\tau_k \in T(B^\infty)$ such that for every $\tau \in T(B^\infty)$,
\begin{equation} \min\{\tau(a_{\tau_1}),\dots,\tau(a_{\tau_k})\} < \e^2. \end{equation}
Using CPoU as in Lemma \ref{lem:CPoU-Infinity}, there exist orthogonal projections $e_1,\dots,e_k \in B^\infty \cap (\psi(\mathcal F)\cup\phi(\mathcal F) \cup \{u_{\tau_1},\dots,u_{\tau_k}\})'$ which sum to $1_{B^\infty}$ such that
\begin{equation}
\label{eq:Uniqueness2}
 \tau(a_{\tau_i}e_i) \leq \e^2\tau(e_i),\quad \tau \in T_\infty(B). \end{equation}
Set
\begin{equation} u:=\sum_{i=1}^k e_iu_{\tau_i}. \end{equation}
Since the $e_i$ are orthogonal projections summing to $1_{B^\infty}$ and using the fact that they commute with the unitaries $u_{\tau_j}$, it follows that $u$ is itself a unitary.
Moreover, for $x \in \mathcal F$ and $\tau \in T_\infty(B)$, using the fact that the $e_i$ are orthogonal projections which commute with the $u_{\tau_i}$, and both $\phi(x)$ and $\psi(x)$, we have
\begin{equation}
\label{eq:Uniqueness3} |\phi(x)-u^*\psi(x)u|= \sum_{i=1}^k e_i|\phi(x)-u_{\tau_i}^*\psi(x)u_{\tau_i}|. \end{equation}
Hence
\begin{eqnarray}
\notag
\|\phi(x)-u^*\psi(x)u\|_{2,\tau}^2
&\stackrel{\eqref{eq:Uniqueness3}}=& \sum_{i=1}^k \tau\big(e_i|\phi(x) - u_{\tau_i}^*\psi(x)u_{\tau_i}|^2\big) \\
\notag
&\stackrel{\eqref{eq:Uniqueness1}}\leq& \sum_{i=1}^k \tau(e_ia_{\tau_i}) \\
&\stackrel{\eqref{eq:Uniqueness2}}\leq& \sum_{i=1}^k \e^2\tau(e_i) = \e^2.
\end{eqnarray}
Taking the supremum over all $\tau \in T_\infty(B)$, \eqref{eq:Uniqueness0} follows.
\end{proof}

In order to get our existence result into $B^\infty$, we begin with two existence results into von Neumann algebras; these rely on the quasidiagonality of amenable traces on cones established in \cite{BCW16}, which in turn builds on the earlier results of \cite{SWW15,Ga16}.  Recall that a trace $\tau$ on $A$ is \emph{amenable} if given a finite subset $\mathcal F\subset A$ and $\epsilon>0$ there is a c.p.c.\ map $\phi:A\to M_n$ for some $n$ ($\phi$ can be taken to be unital when $A$ is unital) such that \begin{equation}\label{neweq}
\|\phi(ab)-\phi(a)\phi(b)\|_{2,\tr_{M_n}}<\epsilon,\quad a,b\in \mathcal F,
\end{equation}
and
\begin{equation}
|\tr_{M_n}(\phi(a))-\tau(a)|<\epsilon,\quad a\in \mathcal F.
\end{equation}
We write $T_{\mathrm{am}}(A)$ for the set of amenable traces on $A$.
The trace $\tau$ is said to be \emph{quasidiagonal} if (\ref{neweq}) can be strengthened to the operator norm estimate $\|\phi(ab)-\phi(a)\phi(b)\|<\epsilon$ for $a\in\mathcal F$.  We write $T_{\mathrm{qd}}(A)$ for the set of quasidiagonal traces on $A$.  See \cite{Bro06} for details on these approximation properties.

\begin{lemma}\label{l:existence1}
	Let $A$ be a $C^*$-algebra, let $\mathcal M$ be a type II$_1$ von Neumann algebra, and let $\lambda \in T_{\mathrm{am}}(A)$. 
	Then given a finite set $\mathcal{F} \subset A$ and $\epsilon>0$, there exist a finite dimensional $C^*$-algebra $F$, a c.p.c.\ map $\theta:A \to F$, and a unital $^*$-homomorphism $\eta:F \to \mathcal M$ such that:
	\begin{align}
	\|\theta(a)\theta(b)\| &< \epsilon \quad && \text{for }a,b \in \mathcal F \text{ satisfying }ab=0\text{, and} \label{l:existence1:eq1} \\
	|\tau (\eta \circ \theta(a))-\lambda(a)|&<\epsilon &&\text{for }a \in \mathcal F\text{ and }\tau \in T(\mathcal M). \label{l:existence1:eq2}
	\end{align}
\end{lemma}

\begin{proof}
	Set $\mathcal G : = \{ \mathrm{id}_{(0,1]} \otimes a \in C_0((0,1]) \otimes A : a \in \mathcal{F} \}$.
	By \cite[Proposition 3.2]{BCW16}, the trace $\delta_1 \otimes \lambda$ is quasidiagonal on $C_0((0,1]) \otimes A$, where $\delta_1$ is the functional of evaluation at $1$ on $C_0((0,1]$.  Thus, there exist a matrix algebra $F$ and a c.p.c.\ map $\phi: C_0((0,1]) \otimes A \to F$ such that 
	\begin{align}
\notag
	\| \phi(x) \phi(y) - \phi(xy) \| & < \varepsilon,\quad x,y \in \mathcal G,  \\ 
	|\tr_F \circ \phi ( x ) - (\delta_1 \otimes \lambda)(x)| &< \varepsilon, \quad x \in \mathcal G. 
	\end{align}
	Define $\theta: A \to F$ by $\theta(a):=\phi(\mathrm{id}_{(0,1]} \otimes a)$, so that it immediately follows that \eqref{l:existence1:eq1} is satisfied. As $\mathcal M$ is type II$_1$, $F$ can be embedded unitally in $\mathcal M$. Let $\eta:F \hookrightarrow\mathcal  M$ be any such embedding. By the uniqueness of the trace on $F$, \eqref{l:existence1:eq2} is also satisfied. 
\end{proof}

\begin{lemma}\label{l:existence2}
	Let $A$ be a $C^*$-algebra, let $\mathcal M$ be a type II$_1$ von Neumann algebra %with a faithful normal trace, 
and let $\alpha:T(\mathcal M) \to T_{\mathrm{am}}(A)$ be affine and continuous.
	Let $\mathcal{F} \subset A_{sa}$ be a finite set.
	Then given $\epsilon>0$, there exist a finite dimensional $C^*$-algebra $F$, a c.p.c.\ map $\theta:A \to F$, and a unital $^*$-homomorphism $\eta:F \to \mathcal M$ such that
	\begin{equation}
	\label{eq:existence2a}
	\|\theta(a)\theta(b)\| < \epsilon \quad  \text{for }a,b \in \mathcal{F} \text{ satisfying }ab=0
	\end{equation} 
	and
	\begin{equation}
	\label{eq:existence2b}
	|\tau(\eta\circ\theta(a))- \alpha(\tau)(a)|<\epsilon, \quad a \in \mathcal F,\ \tau\in T(\mathcal M).
	\end{equation}
	Moreover, if for each $a \in \mathcal{F}$, we are given an element $c_a \in \mathcal M_{sa}$ satisfying
	\begin{equation}\label{eq:existence:neweq} \tau(c_a) = \alpha(\tau)(a), \quad \tau \in T(\mathcal M), \end{equation}
	then for each $a \in \mathcal F$ there exist $x^{(a)}_1,\dots,x_{10}^{(a)},y_1^{(a)},\dots,y_{10}^{(a)}\in \mathcal M$  such that
	\begin{equation}
	\label{eq:existence2c}
	\|\eta \circ \theta(a)-c_a- \sum_{i=1}^{10}[x_i^{(a)},y_i^{(a)}]\| < \epsilon.
	\end{equation}
	If $\alpha(T(\mathcal M))\subset T_{\mathrm{qd}}(A)$, then $\theta$ can be taken to satisfy
	\begin{equation}
	\label{eq:existence2d}
	\|\theta(a)\theta(b) - \theta(ab)\| < \epsilon \quad  \text{for }a,b \in \mathcal{F}.
	\end{equation} 	
\end{lemma}

\begin{proof}
	The idea is to glue maps from the previous lemma over the centre $Z(\mathcal M)$ of $\mathcal M$, in a manner similar to the proof of \cite[Lemma 2.5]{BCW16}.
	As for $a\in\mathcal F$, the elements $c_a$ satisfying \eqref{eq:existence:neweq} automatically exist by Proposition \ref{prop:CP}, we shall use them throughout the proof.
	Suppose $\mathcal{F}=\{a_1, \ldots, a_n \}$.
	By the structure of commutative von Neumann algebras (\cite[Theorem III.1.5.18]{Bla06}), let $(X,\mu)$ be a locally finite measure space such that $Z(\mathcal M) \cong L^\infty(X, \mu)$.
%\footnote{This is possible as $\mathcal M$ has a faithful normal trace.} 
Let $E: \mathcal M \to L^\infty(X, \mu)$ be the centre-valued trace on $\mathcal M$ (see \cite[Theorem III.2.5.7]{Bla06}).
	Choose natural numbers $C\geq 4$ and $k$ such that $C > \sup_{a \in \mathcal{F}} \|c_a\|$ and $C/k < \epsilon$. 
	Set $I :=  \{-Ck+1, \ldots, Ck\}^{n}$ and for $r=(r_1, \ldots, r_n) \in I$, let $p_r \in L^\infty(X,\mu)$ be the characteristic function of the set
	\begin{equation}
	\{ x \in X :  \frac{r_j -1}{k} \leq E(c_{a_j})(x) < \frac{r_j}{k}, \, j=1, \ldots, n\}.
	\end{equation}
	By construction, $(p_r)_{r\in I}$ forms a partition of unity consisting of projections, and, as every trace on $\mathcal M$ factors though $E$ (\cite[Theorem III.2.5.7(iv)]{Bla06}),
	\begin{equation}
	\tau(c_{a_j}) \approx_{1/k} \sum_{r \in I} \frac{r_j}{k} \tau(p_r), \qquad \tau \in T(\mathcal M).\footnote{To improve the readability of this proof, we write $z_1 \approx_\eta z_2$ as shorthand for $|z_1 - z_2| \leq \eta$.} 
	\end{equation}
	In particular, for any $r \in I$ and $j = 1, \ldots, n$, we have
	\begin{equation}
	\alpha(\tau)(a_j) \stackrel{\eqref{eq:existence:neweq}}= \tau(c_{a_j}) \approx_{1/k} \frac{r_j}{k}, \qquad \tau \in T(p_r \mathcal M). \label{l:existence2:eq1}
	\end{equation}
	(Note that we implicitly extend $\tau$ to $\mathcal M$, by setting it to be zero on $(1-p_r) \mathcal M$, before we apply $\alpha$ in the previous equation.)
	
	Let $I_0: =\{r\in I:p_r\neq 0\}$. For each $r\in I_0$, fix $\sigma_r \in T(p_r \mathcal M)$ and set $\lambda_r := \alpha(\sigma_r)$.
	By Lemma \ref{l:existence1}, applied to $p_r \mathcal M$ and $\lambda_r$, there exist a finite dimensional algebra $F_r$, a c.p.c.\ map $\theta_r: A \to F_r$, and a unital $^*$-homomorphism $\eta_r:F_r \to p_r \mathcal M$ such that 
	\begin{equation}
	\|\theta_r(a)\theta_r(b)\| < \epsilon
	\end{equation}
	for $a,b \in \mathcal{F}$ satisfying $ab=0$, and
	\begin{equation}
	|\tau(\eta_r \circ \theta_r(a))-\lambda_r(a)|<\frac{\epsilon}{2}  \label{l:existence2:eq2}
	\end{equation}
	for $a \in \mathcal{F}$ and $\tau \in T(p_r \mathcal M)$.  
	Set $F:=\bigoplus_{r \in I_0} F_r$. Define $\theta: A \to F$ by $\theta(a) := \oplus_{r \in I_0} \theta_r (a)$ and $\eta: F \to\mathcal  M$ by $\eta((x_r)_{r \in I_0}) := \sum_{r \in I_0} \eta_r (x_r)$. By construction, $\eta$ is a unital $^*$-homomorphism and $\theta$ satisfies \eqref{eq:existence2a}.
	
	Note that if each $\lambda_r$ is quasidiagonal, this can be used directly in place of Lemma \ref{l:existence1} in the previous paragraph enabling $\theta_r$ to be chosen $(\mathcal{F},\e)$-approximately multiplicative. Therefore, if $\alpha(T(\mathcal M))\subset T_{\mathrm{qd}}(A)$, then $\theta$ can be taken to satisfy \eqref{eq:existence2d}.
	
	Fix $\tau \in T(\mathcal M)$ for the moment.
	For each $r\in I_0$, set $\tau_r := \frac{\tau(p_r \, \cdot )}{\tau(p_r)} \in T(p_r \mathcal M)$,\footnote{We can choose $\tau_r$ arbitrarily in case $\tau(p_r)=0$.} so $\tau$ can expressed as the convex combination
	\begin{equation}
	\tau = \sum_{r \in I_0} \tau(p_r) \tau_r. \label{l:existence2:eq3}
	\end{equation}
	Thus
	\begin{eqnarray}
	\notag
	\tau (\eta \circ \theta (a)) \hspace*{-5mm} &=& \sum_{r \in I_0} \tau(p_r) \tau_r (\eta_r \circ \theta_r (a)) \\
	\notag
	& \stackrel{\eqref{l:existence2:eq2}}{\capprox{\epsilon/2}}& \sum_{r \in I_0} \tau(p_r) \lambda_r (a) \\
	\notag
	& \stackrel{\eqref{l:existence2:eq1}}{\capprox{2/k}}&  \sum_{r \in I_0} \tau(p_r) \alpha(\tau_r) (a) \\
	\notag
	&=&\alpha(\sum_{r \in I_0} \tau(p_r) \tau_r) (a) \\
	& \stackrel{\eqref{l:existence2:eq3}}{=}& \alpha (\tau)(a) \label{l:existence2:eq4} 
	\end{eqnarray}
	for all $a \in \mathcal{F}$.
	Since $\frac2k\leq\tfrac C{2k}<\tfrac\epsilon2$ and $\tau \in T(\mathcal M)$ was arbitrary, this establishes \eqref{eq:existence2b}.
	
	Now fix $a \in \mathcal{F}$ for the moment and let us explain why $x^{(a)}_i,y^{(a)}_i$ can be found to satisfy \eqref{eq:existence2c}.
	Set $h:= E(\eta \circ \theta (a) - c_a ) \in L^\infty(X,\mu)$, which by \eqref{eq:existence2b} satisfies $\|h\|\leq \epsilon$. Observe that $E(\eta \circ \theta (a) - c_a - h)= 0$, so that by \cite[Theorem 3.2]{FH80} there exist $x^{(a)}_1,\dots,x_{10}^{(a)},y_1^{(a)},\dots,y_{10}^{(a)}\in \mathcal M$ such that\footnote{We can also arrange that $\max_{1\leq i \leq 10}\|x^{(a)}_i\|\|y^{(a)}_i\| \leq 12 \cdot 12 \cdot \|\eta \circ \theta (a) - c_a - h\|$, but we do not need to control the norms of these elements on this occasion.} 
	\begin{equation}
	\eta \circ \theta (a) - c_a - h = \sum_{i=1}^{10} [x^{(a)}_i,y^{(a)}_i].
	\end{equation}
Hence,
	\begin{equation}
	\| \eta \circ \theta (a) - c_a - \sum_{i=1}^{10} [x^{(a)}_i,y^{(a)}_i] \| = \| h \| <  \epsilon.\qedhere
	\end{equation}
\end{proof} 

We now turn to the existence component of Theorem \ref{thm:MainThm}. For this we will need a so-called `no silly traces' result to show that the limit traces on $B^\infty$ generate all traces on $B^\infty$.  For the purposes of Theorem \ref{thm:MainThm}, we could use (a sequence algebra version) of the original result of this type: \cite[Theorem 8]{Oz13} for $\mathcal Z$-stable exact $C^*$-algebras $B$. This gives a no silly traces result for the $C^*$-algebra ultraproduct, from which it follows that there are no silly traces on the uniform tracial ultraproduct (and this is easily modified to sequence algebras).  

However, in Theorem \ref{thm:Existence} below, we prefer not to impose the hypothesis that $B$ is $\mathcal Z$-stable, and instead simply ask that it has CPoU.
Correspondingly, we first show how to obtain a no silly traces result for the uniform tracial sequence algebra just assuming CPoU.  While CPoU is involved to handle possibly non-$\mathcal Z$-stable C*-algebras, to some extent the present result is easier than Ozawa's \cite[Theorem 8]{Oz13} in that we can use the uniform bounds on the number of commutators of a self-adjoint operator in a finite von Neumann algebra which vanishes in all traces from \cite{FH80}, rather than the more delicate growth rate estimates used in \cite{Oz13} which are required to eliminate silly traces from the $C^*$-norm sequence algebra or ultrapower.

We note also that no silly traces for the tracial product $B^\infty$ does not imply no silly traces for the norm product, as demonstrated by the unique trace example by Robert (based on earlier examples by Villadsen) in \cite[Theorem 1.4]{Ro15}.\footnote{Let $A$ be the $C^*$-algebra from \cite[Theorem 1.4]{Ro15} so that condition (iii)(a) from \cite[Theorem 3.24]{ART} fails for $A$.  Then no norm ultraproduct of $A$ can have unique trace (by the equivalence of condition (iii)(a) and the uniqueness of trace on an ultraproduct, which makes up the first part of the proof of (iii)$\Leftrightarrow$(ii) in \cite[Theorem 3.24]{ART}; see paragraph 3 of the proof, which notes this explicitly). Note that $A$ has the Diximer property needed to apply this result by \cite{HZ} as it is simple, unital and has unique trace. Therefore the norm product $A_\infty$ of infinitely many copies of $A$ has traces which are not in the closed convex hull of the limit traces.}

\begin{proposition}\label{prop:nosillytraces}
Let $B$ be a separable $C^*$-algebra with $T(B)$ compact and non-empty and which has CPoU.  Then the weak$^*$-closed convex hull of $T_\infty(B)$ is $T(B^\infty)$.
\end{proposition}
\begin{proof}
Fix a self-adjoint contraction $z\in B^\infty$. Let $\delta:= \sup_{\tau\in T_\infty(B)}|\tau(z)|$. By \cite[Lemma 4.4]{CETWW} (which is extracted from \cite[Theorem 8]{Oz13}), it suffices to prove that $\sup_{\tau\in T(B^\infty)}|\tau(z)|=\delta$.  This will be achieved by producing a self-adjoint $c\in B^\infty$ with $\|c\|\leq\delta$ and contractions $x^{(1)},\dots,x^{(10)},y^{(1)},\dots,y^{(10)}\in B^\infty$ such that $z-c=K\sum_{i=1}^{10}[x^{(i)},y^{(i)}]$, where $K:=12\cdot 12(1+\delta)$.  

Choose a representative sequence $(z_n)_{n=1}^\infty$ of self-adjoint contractions for
 $z\in B^\infty$. Then $\limsup_{n\to\infty}\sup_{\tau\in T(B)}|\tau(z_n)|\leq\delta$, and so by rescaling, we may assume $\sup_{\tau\in T(B)}|\tau(z_n)|\leq\delta$ for each $n$.  

Fix $n\in\mathbb N$ for the moment. For each $\tau\in T(B)$, let $\pi_\tau$ be its GNS-representation and $\mathcal M_\tau:=\pi_\tau(B)''$. 
Then $\sup_{\rho\in T(\mathcal M_\tau)}|\rho(\pi_\tau (z_n))|\leq\delta$. Letting $\tilde{c}_{n,\tau}\in\mathcal M_\tau$ be the result of applying the centre-valued trace in $\mathcal M_\tau$ to $\pi_\tau(z_n)$, we have $\|\tilde{c}_{n,\tau}\|\leq\delta$.  By \cite[Theorem 3.2]{FH80}, there exist contractions $\tilde{x}_{n,\tau}^{(1)},\dots,\tilde{x}^{(10)}_{n,\tau},\tilde{y}_{n,\tau}^{(1)},\dots,\tilde{y}^{(10)}_{n,\tau}\in\mathcal M_\tau$ with $\pi_\tau(z_n)-\tilde{c}_n=K\sum_{i=1}^{10}[\tilde{x}^{(i)}_{n,\tau},\tilde{y}^{(i)}_{n,\tau}]$.  By Kaplansky's density theorem, there exists a self-adjoint $c_{n,\tau} \in B$ with $\|c_{n,\tau}\|\leq\delta$ and contractions $x^{(1)}_{n,\tau},\dots,x^{(10)}_{n,\tau},y^{(1)}_{n,\tau},\dots,y^{(10)}_{n,\tau}\in B$ with
\begin{equation}
\|z_n-c_{n,\tau}-K\sum_{i=1}^{10}[x^{(i)}_{n,\tau},y^{(i)}_{n,\tau}]\|_{2,\tau} < \gamma_n,
\end{equation}
where $\gamma_n < \tfrac1n$.

Let $a_{n,\tau}:=|z_n-c_{n,\tau}-K\sum_{i=1}^{10}[x^{(i)}_{n,\tau},y^{(i)}_{n,\tau}]|^2$. By compactness, there exist $\tau_{n,1},\dots,\tau_{n,k_n}$ such that $\min_{\rho\in T(B)}\{\rho(a_{n,\tau_{n,1}}),\dots,\rho(a_{n,\tau_{n,k_n}})\}<\gamma_n^2$.  As every trace in $T_\infty(B)$ restricts to a trace on $B$, the same minimum holds over $\rho\in T_\infty(B)$. Let $S_n \subset B^\infty$ be the separable subalgebra generated by $z_n$ together with $c_{n,\tau_{n,1}},\dots,c_{n,\tau_{n,k_n}}$ and the contractions $x^{(1)}_{n,\tau_{n,1}},\dots,x^{(10)}_{n,\tau_{n,k_n}}$ and $y^{(1)}_{n,\tau_{n,1}},\dots,y^{(10)}_{n,\tau_{n,k_n}}$.
 
By CPoU in the form of Lemma \ref{lem:CPoU-Infinity}, there exist pairwise orthogonal projections $e_{n,1},\dots,e_{n, k_n}$ in $B^\infty \cap S_n'$ which sum to $1_{B^\infty}$ and have $\rho(a_{n, \tau_{n,j}} e_{n,j})\leq \gamma_n^2 \rho(e_{n,j})$ for $j=1,\dots,k_n$ and all $\rho\in T_\infty(B)$. 

Define $\tilde{c}_n:=\sum_{j=1}^{k_n}c_{n,\tau_{n,j}}e_{n,j} \in B^\infty$, $\tilde{x}^{(i)}_n:=\sum_{j=1}^{k_n}x^{(i)}_{n,\tau_{n,j}}e_{n,j}  \in B^\infty$, and $\tilde{y}^{(i)}_n:=\sum_{j=1}^{k_n}y^{(i)}_{n,\tau_{n,j}}e_{n,j}  \in B^\infty$. Then $\|\tilde{c}_n\| \leq \delta$ and all the $\tilde{x}^{(i)}_n$ and $\tilde{y}^{(i)}_n$ are contractions. 
Let $\rho\in T_\infty(B)$. Using the properties of the $e_{n,j}$, we have
\begin{align}
\|z_n-\tilde{c}_{n}-K\sum_{i=1}^{10}[\tilde{x}^{(i)}_{n},\tilde{y}^{(i)}_{n}]\|_{2,\rho}^2=\rho(\sum_{j=1}^{k_n}a_{n,\tau_{n,j}}e_{n,j}) \leq \gamma_n^2.
\end{align}

Taking norm preserving lifts from $B^\infty$ to $\ell^\infty(B)$ and then choosing elements $c_n$, $x^{(i)}_n$, $y^{(i)}_n \in B$ sufficiently far down the representative sequences for $\tilde{c}_n, \tilde{x}^{(i)}_n,\tilde{y}^{(i)}_n \in B^\infty$, we have
\begin{equation}
\sup_{\tau\in T(B)}\|z_n-c_n-K\sum_{i=1}^{10}[x^{(i)}_n,y^{(i)}_n]\|_{2, \tau} \leq \gamma_n
\end{equation}
and $\|c_n\| \leq \delta$.  Assembling these into $c:=(c_n)_{n=1}^\infty$, $x^{(i)}:=(x^{(i)}_n)_{n=1}^\infty$ and $y^{(i)}:=(y_n^{(i)})_{n=1}^\infty$ in $B^\infty$ provides the elements demanded in the first paragraph of the proof.
\end{proof}

We can now give our more general version of the existence aspect of Theorem \ref{thm:MainThm}. The condition that $A$ be nuclear is weakened to the condition that the range of $\alpha$ is contained in the set $T_{\mathrm{am}}(A)$ of amenable traces on $A$.  The second part of the following theorem, regarding the form of a representative sequence for $\phi$, is not needed for Theorem \ref{thm:MainThm}, but we anticipate it playing a role in future nuclear dimension computations. Recall that a c.p.\ map $\phi:A \rightarrow B$ is said to be \emph{order zero} if $\phi(a)\phi(b) = 0$ for all $a,b \in A_+$ with $ab = 0$; 
see \cite{WZ09} for the structure theory of these maps.

\begin{theorem}
\label{thm:Existence}
Let $A$ be a separable $C^*$-algebra and let $B$ be a separable $C^*$-algebra with $T(B)$ compact and non-empty which has CPoU and no finite dimensional representations. Given a continuous affine function $\alpha:T(B^\infty) \to T_{\mathrm{am}}(A)$, there exists a $^*$-homomorphism $\phi:A \to B^\infty$ such that
\begin{equation}
\label{eq:Existence1}
 \tau \circ \phi = \alpha(\tau), \quad \tau \in T(B^\infty). \end{equation}
Moreover, $\phi$ can be represented by a sequence $(\phi_n)_{n=1}^\infty$ of c.p.c.\ maps $A\rightarrow B$ each of which factorises as $\phi_n=\psi_n\circ\theta_n$ for a c.p.c.\ map $\theta_n:A\rightarrow F_n$ with $F_n$ finite dimensional, and a c.p.c.\ order zero map $\psi_n:F_n\rightarrow B$.  The maps $\theta_n$ can be taken to be approximately order zero, and if the range of $\alpha$ lies in the quasidiagonal traces on $A$, the $\theta_n$ can be taken to be approximately multiplicative.
\end{theorem}

\begin{proof}
We first note that any c.p.c.\ order zero map $\phi:A \to B^\infty$ satisfying \eqref{eq:Existence1} is automatically a $^*$-homomorphism.
Indeed, let $(e_n)_{n=1}^\infty$ be an increasing approximate unit for $A$. Then 
\begin{equation}\label{eq:OZ}
	\phi(a_1)\phi(a_2) = \lim_{n\rightarrow\infty} \phi(e_n)\phi(a_1a_2), \quad a_1,a_2 \in A, 
\end{equation} 
as a consequence of \cite[Corollary 4.1]{WZ09}. It therefore suffices to prove that $\lim_{n\rightarrow\infty} \phi(e_n) = 1_{B^\infty}$ in $\|\cdot\|_{2,T_\infty(B)}$. We compute that 
\begin{align}
\notag
	\|1_{B^\infty}-\phi(e_n)\|_{2,T_\infty(B)}^2 &\leq \sup_{\tau \in T(B^\infty)} \tau(|1_{B^\infty}-\phi(e_n)|^2)\\
\notag
	&\leq \sup_{\tau \in T(B^\infty)} \tau(1_{B^\infty}-\phi(e_n))\\
\notag
	&\leq \sup_{\tau \in T(B^\infty)} 1 - \tau(\phi(e_n))\\
	&= \sup_{\tau \in T(B^\infty)} (1 - \alpha(\tau)(e_n))\rightarrow 0,
\end{align}
as by Dini's Theorem, $\alpha(\tau)(e_n)$ converges to 1 uniformly on $T(B^\infty)$.

Notice that due to Proposition \ref{prop:nosillytraces}, it is enough to establish equation \eqref{eq:Existence1} for limit traces.
Fix $\e>0$ and finite sets $\mathcal F\subset A$ and $\mathcal G \subset A_{sa}$. We will prove that there is a finite dimensional $C^*$-algebra $F$, a c.p.c.\ map $\theta:A\rightarrow F$ and a c.p.c.\ order zero map $\psi:F\rightarrow B^\infty$ such that for $\phi=\psi\circ\theta$ we have
\begin{align}
	\|\theta(a)\theta(b)\| &\leq \epsilon \quad && \text{for }a,b \in \mathcal F \text{ satisfying }ab=0\text{, and} \label{eq:Existence2a} \\
	|\tau (\phi(a))-\alpha(\tau)(a)|&\leq \epsilon &&\text{for }a \in \mathcal G \text{ and }\tau \in T_\infty(B). \label{eq:Existence2b}
\end{align}
In the special case that $\alpha(T(B^\infty))\subseteq T_{\mathrm{qd}}(A)$, we will show that we can additionally replace (\ref{eq:Existence2a}) by the stronger condition
\begin{equation}\label{eq:Existence2c}
\|\theta(a)\theta(b)-\theta(ab)\|\leq\epsilon,\quad a,b\in\mathcal F.
\end{equation}
Once this is achieved, an application of Kirchberg's $\e$-test (in the form of Lemma \ref{lem:EpsTest}) can be used to obtain the required $\phi$ (and $\psi_n,\theta_n$) in a very similar fashion to \cite[Lemma 7.4]{BBSTWW}. We set this out for the passage from (\ref{eq:Existence2a}) and (\ref{eq:Existence2b}) to the required $\phi$ such that $(\theta_n)$ are approximately order zero in  the next paragraph. The passage from (\ref{eq:Existence2c}) and (\ref{eq:Existence2b}) to obtaining a $\phi$ such that the maps $(\theta_n)_{n=1}^\infty$ are approximately multiplicative is similar (and slightly easier).

For each $n$, let $X_n$ denote the set of triples $(F_n,\theta_n,\psi_n)$, where $F_n$ is a finite dimensional $C^*$-algebra $F_n$, $\theta_n:A\to F_n$ is c.p.c.\ and $\psi_n:F_n\to B$ is c.p.c.\ order zero.\footnote{Although all the $X_n$ represent the same set, we use the subscript $n$ for direct comparison with the notation of the $\e$-test.}  Fix a countable dense subset $(x_k)_{k=1}^\infty$ of $A_{\mathrm{sa}}$. Noting that the collection of pairs $(a,b)$ in $A_+$ with $ab=0$ is a subspace of the separable metric space $A \times A$, we may also fix a countable dense subset $(a_k,b_k)_{k=1}^\infty$ of these orthogonal pairs.  Set $X:=\prod_{n=1}^\infty X_n$ and define functions $f^{(k)}_n:X_n\to[0,\infty]$ by
\begin{align}
\notag
f^{(k)}_n(F_n,\theta_n,\psi_n):=\max_{j\leq k}&\big(\|\theta_n(a_jb_j)\| \\
&+\sup_{\tau\in T(B)}|\tau(\psi_n(\theta_n(x_j)))-\alpha(\tau(x_j))|\big),
\end{align}
for $(F_n,\theta_n,\psi_n)\in X_n$.  Given $\e>0$ and $k_0\in\mathbb N$, let $\mathcal{F}:=\{a_k,b_k:k\leq k_0\}$ and $\mathcal G:=\{x_k:k\leq k_0\}$ and take $F,\theta,\psi$ satisfying (\ref{eq:Existence2a}) and \eqref{eq:Existence2b}.  Then $\psi$ lifts to a sequence $(\tilde\psi_n)_{n=1}^\infty$ of c.p.c.\ order zero maps $F\to B$ by projectivity of c.p.c.\ order zero maps with finite dimensional domains (\cite[Proposition 1.2.4]{Wi09}, which rephrases Loring's work \cite[Theorem 4.9]{Lo93} on projectivity of cones over finite dimensional $C^*$-algebras to this setting).  The sequence $\big(F,\theta,\tilde\psi_n\big)_{n=1}^\infty$ in $\prod_{n=1}^\infty X_n$ satisfies
\begin{equation}
\limsup_n f^{(k)}_n\big(F,\theta,\tilde\psi_n\big)\leq 2\epsilon,\quad k=1,\dots,k_0.
\end{equation}
Applying the $\e$-test gives a sequence $(F_n,\theta_n,\psi_n)$ in $\prod_{n=1}^\infty X_n$ with 
\begin{equation}
\limsup_nf^{(k)}_n(F_n,\theta_n,\psi_n)=0,\quad k\in\mathbb N.
\end{equation} 
Defining $\phi_n:=\psi_n\circ\theta_n$ and $\phi:A\to B^\infty$ to be the map induced by $(\phi_n)_{n=1}^\infty$ gives the required $\phi$.

With the $\epsilon$-test in place, we now commence the construction of maps satisfying (\ref{eq:Existence2a}) and \eqref{eq:Existence2b}. By Proposition \ref{prop:CP}, for each $a \in \mathcal G$, we may choose a self-adjoint element $c_a \in B^\infty$ such that
\begin{equation} \label{eq:existence.new}
\tau(c_a) = \alpha(\tau)(a), \quad \tau \in T(B^\infty).
\end{equation}
Also, set \begin{equation} \delta := \frac\e{\sqrt{|\mathcal G|}}. \end{equation}

We work with the weak$^*$-closure $\overline{T_\infty(B)}$ of $T_\infty(B)$ which is weak$^*$-compact. Fix $\tau \in \overline{T_\infty(B)}$ for the moment. Consider the finite von Neumann algebra $\mathcal M_\tau:=\pi_\tau(B^\infty)''$.

We claim that $\mathcal M_\tau$ is type II$_1$. As $T(B)$ is compact, any approximate unit $(e_n)_{n=1}^\infty$ for $B$ satisfies $\inf_{\rho\in T(B)}\rho(e_n)\rightarrow 1$ by Dini's Theorem.  Then 
\begin{equation}
\inf_{\rho\in T_\infty(B)}\rho(e_n)=\inf_{\rho\in \overline{T_\infty(B)}}\rho(e_n)\rightarrow 1,
\end{equation} 
and hence $\tau(e_n)\rightarrow 1$.  Thus $\|1_{\mathcal M_\tau}-\pi_\tau(e_n)\|_{2,\tau}\rightarrow 0$ and $\pi_\tau (e_n)$ converges $^*$-strongly to $1_{\mathcal{M}_\tau}$. Suppose that $\mathcal M_\tau$ has a non-zero type I$_k$ summand for some $k\in\mathbb N$, with corresponding central projection $p$, so that $p\mathcal M_\tau$ has a separating family of finite dimensional representations.  Since $B$ has no non-zero finite dimensional representations, we must have $p\pi_\tau(B)=0$.  On the other hand $p\pi_\tau(e_n)\neq 0$ for sufficiently large $n$. This contradiction proves the claim.

Let $T(\pi_\tau): T(\mathcal{M}_\tau) \to T(B^\infty)$ be the map induced by $\pi_\tau$.
By Lemma \ref{l:existence2} (applied to the map $\alpha\circ T(\pi_\tau):T(\mathcal M_\tau)\rightarrow T_{\mathrm{am}}(A)$), there exist a finite dimensional $C^*$-algebra $F_\tau$, a c.p.c.\ map $\theta_\tau:A \to F_\tau$ and a unital $^*$-homomorphism $\eta_\tau:F_\tau \to \mathcal M_\tau$ as well as $x^{(a)}_1,\dots,x^{(a)}_{10},y^{(a)}_1,\dots,y^{(a)}_{10} \in \mathcal M_\tau$ for $a \in \mathcal G$, such that
\begin{gather}
	\label{eq:Existence3a}
	\|\theta_\tau(a)\theta_\tau(b)\| < \e \quad  \text{for }a,b \in \mathcal{F} \text{ satisfying }ab=0,\text{ and} \\
	\label{eq:Existence3b}
	\|\eta_\tau \circ \theta_\tau(a)-\pi_\tau(c_a)- \sum_{i=1}^{10}[x_i^{(a)},y_i^{(a)}]\| < \delta\quad \text{for }a \in \mathcal{G}.
\end{gather}

By the Kaplansky density theorem, at the cost of replacing the norm estimate in \eqref{eq:Existence3b} by a $\|\cdot\|_{2,\tau}$-estimate, we may assume that the elements $x^{(a)}_i,y^{(a)}_i$ belong to $\pi_\tau(B^\infty)$, and thus lift to elements $x^{(\tau,a)}_i,y^{(\tau,a)}_i \in B^\infty$.
Using the order zero Kaplansky density theorem (\cite[Lemma 1.1]{HKW12}) we may approximate $\eta_\tau$ in $\|\cdot\|_{2,\tau}$ by an order zero map $F_\tau \to \pi_\tau(B^\infty)$, and then this can be lifted to an order zero map $\psi_\tau:F_\tau \to B^\infty$ by \cite[Proposition 1.2.4]{Wi09}.
Starting from \eqref{eq:Existence3b}, we can perform these approximations and lifts so that
\begin{equation}
\label{eq:Existence4b}
	\|\psi_\tau(\theta_\tau(a))-c_a- \sum_{i=1}^{10}[x_i^{(a,\tau)},y_i^{(a,\tau)}]\|_{2,\tau} < \delta,\quad a \in \mathcal{G}. \end{equation}
Let us now set
\begin{equation} s_\tau := \sum_{a \in \mathcal G} \Big|\psi_\tau(\theta_\tau(a))-c_a-\sum_{i=1}^{10} [x_i^{(a,\tau)},y_i^{(a,\tau)}]\Big|^2 \in (B^\infty)_+, \end{equation}
so that by \eqref{eq:Existence4b} we get $\tau(s_\tau) < |\mathcal G|\delta^2 = \e^2$.

By continuity and compactness of $\overline{T_\infty(B)}$, there exist $\tau_1,\dots,\tau_k \in \overline{T_\infty(B)}$ such that for every $\tau \in \overline{T_\infty(B)}$,
\begin{equation} \min\{\tau(s_{\tau_1}),\dots,\tau(s_{\tau_k})\} < \e^2. \end{equation}
Set \begin{align} \notag S&:= \psi_{\tau_1}(F_{\tau_1}) \cup \cdots \cup \psi_{\tau_k}(F_{\tau_k}) \cup \{c_a:a \in \mathcal G\} \\
&\qquad\cup \{x^{(a,\tau_i)}_j,y^{(a,\tau_i)}_j: a \in \mathcal G,\ i=1,\dots,k,\ j=1,\dots,10\}, \end{align}
a separable subset of $B^\infty$. Using CPoU as in Lemma \ref{lem:CPoU-Infinity}, there exist orthogonal projections $e_1,\dots,e_k \in B^\infty \cap S'$ which sum to $1_{B^\infty}$ such that
\begin{equation}
\label{eq:Existence5}
 \tau(s_{\tau_i}e_i) \leq \e^2\tau(e_i),\quad \tau \in T_\infty(B). \end{equation}
Set $F:= \bigoplus_{i=1}^kF_{\tau_i}$, and define $\theta:A\to F$ and $\psi:F\to B^\infty$ by
\begin{equation}
\theta(a):=(\theta_{\tau_1}(a),\dots,\theta_{\tau_k}(a)),\quad a\in A,\label{eq:defTheta}
\end{equation}
and
\begin{equation} \psi(x_1,\dots,x_k):=\sum_{i=1}^k e_i\psi_{\tau_i}(x_i),\quad (x_1,\dots,x_k)\in F. \end{equation}
Then \eqref{eq:Existence2a} is an immediate consequence of \eqref{eq:Existence3a}. Since the $e_i$ are orthogonal positive elements commuting with the images of the c.p.\ order zero maps $\psi_{\tau_i}$, it follows that $\psi$ is c.p.\ and order zero.
Moreover $\psi$ is contractive since $\psi(1_A) \leq \sum_{i=1}^k e_i = 1_{B^\infty}$.

Finally, for $a \in \mathcal G$ and $\tau \in T_\infty(B)$, writing $\phi=\psi\circ\theta$, we compute
\begin{eqnarray}
\notag
&&\hspace*{-6em} |\tau(\phi(a))-\alpha(\tau)(a)|^2 \\
&\stackrel{\eqref{eq:existence.new}}=&|\tau(\psi(\theta(a))-c_a)|^2\notag
\\&=&\Big|\tau\Big(\sum_{i=1}^ke_i\Big(\psi_{\tau_i}(\theta_{\tau_i}(a))-c_a-\sum_{j=1}^{10} [x^{(a,\tau_i)}_j,y^{(a,\tau_i)}_j]\Big)\Big)\Big|^2\notag \\
\notag
&\leq& \tau\Big(\sum_{i=1}^k e_i\Big|\psi_{\tau_i}(\theta_{\tau_i}(a))-c_a-\sum_{j=1}^{10} [x^{(a,\tau_i)}_j,y^{(a,\tau_i)}_j]\Big|^2\Big) \\
\notag
&\leq& \tau\Big(\sum_{i=1}^k e_is_{\tau_i}\Big) \\
&\leq& \sum_{i=1}^k \e^2\tau(e_i)  = \e^2,
\end{eqnarray}
where on the third line, we use the fact that the $e_i$ make up a pairwise othogonal partition of unity of projections commuting with $S$. This proves \eqref{eq:Existence2b}, and completes the proof with $\theta_n$ approximately order zero.

When $\alpha(T_\infty(B))$ consists of quasidiagonal traces, \eqref{eq:Existence3a} can be replaced by an $(\mathcal{F},\e)$-approximate multiplicativity condition using the last clause of Lemma \ref{l:existence2}. The map $\theta$, as defined in \eqref{eq:defTheta}, will then satisfy \eqref{eq:Existence2c}, and this completes the proof with $\theta_n$ approximately multiplicative.
\end{proof}

We end by recording how Theorem \ref{thm:MainThm} follows as special cases of the existence and uniqueness results of this section.
\begin{proof}[Proof of Theorem \ref{thm:MainThm}]
This is a consequence of Theorems \ref{thm:Uniqueness} and \ref{thm:Existence}. The CPoU hypothesis on $B$ needed in both these theorems is automatic for separable nuclear $\mathcal Z$-stable $C^*$-algebras by Theorem \ref{thm:CPoU}; moreover $\mathcal Z$-stability is an obstruction to having finite dimensional representations. The hypothesis in Theorem \ref{thm:Existence} that $\alpha$ takes values in the amenable traces on $A$ is automatic as all traces on a nuclear $C^*$-algebra are amenable, essentially by Connes' theorem.\footnote{In fact, all traces are uniformly amenable by \cite[Theorem 3.2.2(5)$\Rightarrow$(1)]{Bro06} and Connes' theorem, and uniformly amenable traces are amenable: see \cite[Section 3.5]{Bro06}.  (Note that \cite{Bro06} works with unital $C^*$-algebras throughout, so to apply these results one should unitise both $A$ and all its traces).}  
\end{proof}

\newcommand{\etalchar}[1]{$^{#1}$}

\newpage
\appendix

\section{Corrigendum to ``classifying maps into uniform tracial sequence algebras''}

In the proof of Theorem~2.6 of \cite{CETW}, we made an error in Equation~(2.44), where we wrote
\begin{align}
\notag
f^{(k)}_n(F_n,\theta_n,\psi_n):=\max_{j\leq k}&\big(\|\theta_n(a_jb_j)\| \\
\tag{2.44}
&+\sup_{\tau\in T(B)}|\tau(\psi_n(\theta_n(x_j))-\alpha(\tau(x_j))|\big).
\end{align}
There is a typo, and instead of $\alpha(\tau(x_j))$, we had meant to write $\alpha(\tau)(x_j)$.
However this would still be incorrect: since $\alpha$ is a map from $T(B^\infty)$ to $T_{\mathrm{am}}(A)$ and we are quantifying over $\tau \in T(B)$, it does not make sense to take $\alpha(\tau)(x_j)$.

The correction involves representing the affine functionals on $T(B^\infty)$ given by $\tau \mapsto \alpha(\tau)(x_j)$ using self-adjoint elements in $B^\infty$.
This idea, of representing the action of the trace via an element of $B^\infty$, is already used (correctly) later in the proof of \cite[Theorem~2.6]{CETW}.

To obtain a correct proof, we need to define $f^{(k)}_n$ differently.
For each $j\in\mathbb N$, using Proposition~1.2 of \cite{CETW} (which is due to Cuntz and Pedersen), we may find a self-adjoint element $c^{(j)} \in B^\infty$ such that $\tau(c^{(j)})=\alpha(\tau)(x_j)$ for all $\tau \in T(B^\infty)$. For each $c^{(j)}$, we choose a representative sequence $(c^{(j)}_n)_{n=1}^\infty$ of self-adjoint elements in $B$.
We then define 
\begin{align}
\notag
f^{(k)}_n(F_n,\theta_n,\psi_n):=\max_{j\leq k}&\big(\|\theta_n(a_jb_j)\| \\
%\tag{2.44'}
&+\sup_{\tau\in T(B)}|\tau(\psi_n(\theta_n(x_j))-\tau(c^{(j)}_n))|\big).
\end{align}
Then one sees that \cite[(2.45)]{CETW} holds with the new definition of $f^{(k)}_n$, since
\begin{equation}
\begin{array}{rcl}
&& \hspace*{-7em} \limsup_{n\to\infty} \sup_{\tau \in T(B)} |\tau(\tilde\psi_n(\theta(x_j)))-\tau(c^{(j)}_n)| \\
&=& \sup_{\tau \in T_\infty(B)} |\tau(\psi(\theta(x_j))-\tau(c^{(j)})| \\
&=& \sup_{\tau \in T_\infty(B)} |\tau(\psi(\theta(x_j))-\alpha(\tau)(x_j)| \\
&\stackrel{\text{\cite[(2.42)]{CETW}}}\leq& \epsilon.
\end{array} \end{equation} 
Next, given a sequence $(F_n,\theta_n,\psi_n) \in \prod_{n=1}^\infty X_n$ satisfying \cite[(2.46)]{CETW}, define $\phi_n:=\psi_n\circ\theta$ and let $\phi:A \to B^\infty$ be the map induced by $(\phi_n)_{n=1}^\infty$.
Then this is the required $\phi$. Indeed, to prove \cite[(2.38)]{CETW}, it is enough to verify equality on the dense subset $\{x_j:j\in\mathbb N\}$ of $A_{sa}$.
For each $j$, we have
\begin{equation}
\tau\circ\phi(x_j)
= \tau(c^{(j)})  = \alpha(\tau)(x_j),
\end{equation}
for all $\tau \in T_\infty(B)$, hence for all $\tau \in T(B^\infty)$ by \cite[Proposition 2.5]{CETW}, as required.

\section*{Acknowledgments}
We thank Jamie Gabe for bringing this error to our attention.

\end{document}